\documentclass[10pt,reqno]{amsart}

\usepackage{fullpage,amsmath,amsthm,amsfonts,amssymb,hyperref,enumitem,amsaddr}
\usepackage[T1]{fontenc}
\usepackage[utf8]{inputenc}
\usepackage{mathtools}
\usepackage{aliascnt}

\setlist[enumerate]{font=\normalfont}

\numberwithin{equation}{section}

\hypersetup{colorlinks=true,linkcolor=blue,citecolor=blue,urlcolor=blue}

\newtheorem{thm}{Theorem}[section]

\newaliascnt{lem}{thm}
\newtheorem{lem}[lem]{Lemma}
\aliascntresetthe{lem}

\newaliascnt{qst}{thm}
\newtheorem{qst}[qst]{Question}
\aliascntresetthe{qst}

\newaliascnt{cor}{thm}
\newtheorem{cor}[cor]{Corollary}
\aliascntresetthe{cor}

\newaliascnt{prop}{thm}
\newtheorem{prop}[prop]{Proposition}
\aliascntresetthe{prop}

\newtheorem*{mainA}{Theorem A}
\newtheorem*{mainB}{Theorem B}
\newtheorem*{mainC}{Theorem C}

\theoremstyle{definition}
\newaliascnt{rem}{thm}
\newtheorem{rem}[rem]{Remark}
\aliascntresetthe{rem}

\newaliascnt{conv}{thm}
\newtheorem{conv}[conv]{Convention}
\aliascntresetthe{conv}

\newaliascnt{defn}{thm}
\newtheorem{defn}[defn]{Definition}
\aliascntresetthe{defn}

\usepackage[nameinlink,capitalize]{cleveref}
\crefname{thm}{Theorem}{Theorems}
\Crefname{thm}{Theorem}{Theorems}
\crefname{prop}{Proposition}{Propositions}
\Crefname{prop}{Proposition}{Propositions}
\crefname{lem}{Lemma}{Lemmas}
\Crefname{lem}{Lemma}{Lemmas}
\crefname{cor}{Corollary}{Corollaries}
\Crefname{cor}{Corollary}{Corollaries}
\crefname{rem}{Remark}{Remarks}
\Crefname{rem}{Remark}{Remarks}
\crefname{defn}{Definition}{Definitions}
\Crefname{defn}{Definition}{Definitions}
\crefname{qst}{Question}{Questions}
\Crefname{qst}{Question}{Questions}
\crefname{conv}{Convention}{Conventions}
\Crefname{conv}{Convention}{Conventions}

\newcommand{\Z}{\mathbb Z}
\newcommand{\Aut}{\operatorname{Aut}}
\newcommand{\Out}{\operatorname{Out}}
\newcommand{\SAut}{\operatorname{SAut}}
\newcommand{\SOut}{\operatorname{SOut}}
\newcommand{\GL}{\operatorname{GL}}
\newcommand{\SL}{\operatorname{SL}}
\newcommand{\Inn}{\operatorname{Inn}}
\newcommand{\tr}{\operatorname{tr}}
\newcommand{\diag}{\operatorname{diag}}
\newcommand{\eps}{\varepsilon}
\newcommand{\Xpm}{X^{\pm1}}

\title[Nielsen classes in outer automorphism groups]{Nielsen classes in outer automorphism groups of free groups}

\author{Ilya Kapovich}

\address{Department of Mathematics and Statistics, Hunter College of CUNY\newline
  \indent 695 Park Ave, New York, NY 10065, U.S.A.
  \newline \indent e-mail \texttt{ik535@hunter.cuny.edu}
  \newline\indent ORCID 0000-0002-7694-6236
  }

\date{}

\begin{document}

\begin{abstract}
For every $n\ge 8$, we construct explicit families of generating pairs of
$\GL(n,\Z)$ and $\SL(n,\Z)$ representing countably infinitely many Nielsen
equivalence classes.  These pairs arise as the homology images of explicit
torsion generating pairs in $\Out(F_n)$ and its index-two subgroup
$\SOut(F_n)$, yielding countably infinitely many Nielsen classes of generating
pairs in both outer groups.  Within each constructed outer family, all pairs
become Nielsen equivalent after one stabilization, obtained by adjoining a
trivial coordinate.  We derive a relation-module obstruction to Nielsen
equivalence of once-stabilized generating tuples.  Evans's non-cancellation
examples show that this obstruction is nontrivial and can distinguish Nielsen
classes of once-stabilized generating tuples.  We also record a quotient
relation-module refinement for possible future applications.
\end{abstract}

\maketitle

\section{Introduction}\label{sec:introduction}

Recall that for a group $G$ and $n\ge 1$, two $n$-tuples $\mathbf g=(g_1,\dots,g_n)$ and $\mathbf g'=(g_1',\dots,g_n')$ of elements of $G$ are \emph{Nielsen equivalent} in $G$, $\mathbf g\sim_N\mathbf g'$,  if $\mathbf g$ can be transformed to $\mathbf g'$ by a finite chain of \emph{elementary Nielsen moves}: inverting an entry of a tuple, interchanging two entries in a tuple, and multiplying an entry in a tuple on the right by an entry in a different position. See \cref{def:nielsen-equivalence} for a precise definition. If two tuples $\mathbf g=(g_1,\dots,g_n)$ and $\mathbf g'=(g_1',\dots,g_n')$  are Nielsen equivalent in $G$, they generate the same subgroup of $G$, $\langle g_1,\dots, g_n\rangle=\langle g_1',\dots, g_n'\rangle$. For this reason Nielsen equivalence provides a useful tool for studying the subgroup structure of various groups.

The study of Nielsen equivalence goes back to Jakob Nielsen's foundational
work on free groups~\cite{Nielsen,Nielsen1921,Nielsen1924}.  If $F_n=F(x_1,\ldots,x_n)$ and $r\ge n$, Nielsen's
reduction theory implies that every generating $r$-tuple is Nielsen equivalent
to
\[
 (x_1,\ldots,x_n,\underbrace{1,\ldots,1}_{r-n\text{ entries}}).
\]
Thus $F_n$ has a single Nielsen class of generating $r$-tuples for every
$r\ge n$; see \cite{Nielsen,Nielsen1924}.  For a general finitely generated
group $G$, Nielsen equivalence is the orbit relation for
\[
 \Aut(F_r)\curvearrowright\operatorname{Epi}(F_r,G),
\]
and the orbit structure can be much more complicated.

For a finitely generated group $G$, an $n$-tuple $\mathbf g=(g_1,\dots,g_n)\in G^n$ is \emph{generating} if $G=\langle g_1,\dots, g_n\rangle$. For a nontrivial finitely generated group $G$, the smallest $n\ge 1$ such that there exists a generating $n$-tuple for $G$ is called the \emph{rank} of $G$ and denoted $d(G)$. For the trivial group $G=\{1\}$ we set $d(G)=0$.

Studying Nielsen equivalence of generating tuples for finitely generated groups is a challenging but illuminating task. 
Finite groups already exhibit nontrivial phenomena.
B.~H. Neumann and Hanna Neumann found nontrivial examples; in particular,
$A_5$ has three Nielsen classes of generating pairs
\cite{NeumannNeumann1951,McCulloughWanderley}.  In contrast, the Wiegold
conjecture predicts that every generating $n$-tuple of a nonabelian finite
simple group is Nielsen equivalent when $n\ge 3$.  The conjecture is known for several
families, including classical results of Gilman and Evans
\cite{Gilman,EvansFinite}, and is closely related to connectivity questions for
Nielsen graphs and product-replacement graphs
\cite{MyropolskaNagnibeda,AvniGarion}.  Product-replacement algorithms utilize random walks on Nielsen graphs to produce pseudo-random elements of finite groups; see
\cite{PeresTanakaZhai} for quantitative mixing results for the
product-replacement chain and \cite{LeedhamGreenRattle} for a recent variant
whose output converges to the uniform distribution.  For $n=2$, finite simple
groups may have several Nielsen classes and several $T$-systems; see
\cite{EvansFinite,McCulloughWanderley}.

Nielsen methods also play a substantial role in geometric group theory.
Nielsen reduction for groups acting on hyperbolic spaces controls short
generators and the geometry of generating tuples \cite{KWHyperbolic}.  Delzant
\cite{Delzant1991} proved that a torsion-free word-hyperbolic group has only
finitely many conjugacy classes of one-ended two-generated subgroups.  For
finitely generated torsion-free groups, one-ended is equivalent to freely
indecomposable and neither trivial nor infinite cyclic.  Delzant's proof
implicitly gives the corresponding finiteness, up to simultaneous conjugation,
of Nielsen classes of pairs generating such subgroups.  This follows explicitly
from Theorem~1.3 of Kapovich and Weidmann \cite{KWFreely}: with $l=1$ and tuple length $2$,
every such pair is Nielsen equivalent, after simultaneous conjugation, to a pair
in a fixed finite ball.  More generally, if all $l$-generated subgroups of an
almost torsion-free word-hyperbolic group, in the sense of \cite{KWFreely}, are quasiconvex, the same theorem with tuple length
$l+1$ gives finiteness, up to simultaneous conjugation, of Nielsen classes of
$(l+1)$-tuples generating non-elementary freely indecomposable subgroups.  In
particular, a two-generated torsion-free one-ended word-hyperbolic group has
only finitely many ordinary Nielsen classes of generating pairs, since
simultaneous conjugation of a generating pair is induced by an inner
automorphism of $F_2$.  Kapovich and Weidmann later proved that every
torsion-free word-hyperbolic Kleinian group has only finitely many Nielsen
classes of generating $n$-tuples for each fixed $n\ge 1$ \cite{KWKleinian}.

There are also strong uniqueness results.  Extending the classical picture for free
groups, Louder~\cite{Louder} proved that if $G$ is the fundamental group of a
closed surface of negative Euler characteristic, then for every $n\ge d(G)$ any
two generating $n$-tuples of $G$ are Nielsen equivalent.  Kapovich and Schupp
\cite{KapovichSchuppGeneric} constructed exponentially generic classes of
$m$-generator $q$-relator presentations defining torsion-free one-ended
word-hyperbolic groups of rank $m$ for which every $m$-tuple generating a
non-free subgroup is Nielsen equivalent to the given tuple
$(a_1,\ldots,a_m)$.  Thus there is a single Nielsen class of such tuples and the
only non-free $m$-generated subgroup is the whole group.  This Nielsen
uniqueness, combined with generic-case properties of Whitehead's algorithm, is
a key ingredient in the isomorphism-rigidity results of Kapovich, Schupp and
Shpilrain \cite{KapovichSchuppShpilrain}; for generic one-relator groups,
isomorphism implies that the Cayley graphs for the given generating sets are
isometric.  Conversely, Kapovich and Weidmann constructed small-cancellation
examples with generating tuples that remain Nielsen inequivalent after one
stabilization \cite{KWsmall}, and later, for every $n\ge 2$, a torsion-free
one-ended word-hyperbolic group of rank $n$ with two generating $n$-tuples that
remain inequivalent after $n-1$ stabilizations \cite{KWrandom}.  Their
small-cancellation framework also yields finitely presented word-hyperbolic
groups for which Nielsen equivalence of generating tuples is undecidable
\cite{KWsmall}.  These results illustrate the additional difficulty of the
problem for tuples of length at least $3$.

There is a parallel topological motivation from $3$-manifolds.  A Heegaard
splitting determines generating systems of the fundamental group, and isotopic
splittings give Nielsen-equivalent systems.  This has been used to distinguish
Heegaard splittings, beginning with work of Moriah on Seifert fibered spaces
\cite{Moriah1988}.  Lustig and Moriah used Nielsen equivalence in Fuchsian
groups for the same purpose \cite{LustigMoriah1991}, developed
Reidemeister--Whitehead and Nielsen torsion invariants
\cite{LustigMoriahTorsion}, applied related methods to hyperbolic
$3$-manifolds \cite{LustigMoriah1997}, and more recently classified minimal
generating systems in a broad class of Fuchsian groups
\cite{LustigMoriah2022}.

A related algebraic line uses Fox derivatives, relation modules, and algebraic
$K$-theory.  Lustig viewed Nielsen equivalence and simple-homotopy equivalence
as dimension-one and dimension-two analogues and developed invariants that
distinguish Nielsen classes in metabelian and cocompact Fuchsian groups
\cite{Lustig1991}.  This viewpoint is closely related to the relation-module
obstruction to stabilization developed in \cref{sec:relation-modules}.

The case of generating pairs is exceptional.  Nielsen's description of $\Aut(F_2)$ implies
that every automorphism of $F_2=F(x,y)$ sends $[x,y]$ to a conjugate of
$[x,y]^{\pm1}$.  Hence, for every group $G$,
\[
 (g_1,g_2)\sim_N(h_1,h_2)
 \quad\Longrightarrow\quad
 [g_1,g_2]\text{ is conjugate to }[h_1,h_2]^{\pm1}.
\]
B.~H. and Hanna Neumann exploited this rank-two device
\cite{NeumannNeumann1951}; it is closely related to the Higman invariant, and
we refer to it informally as the \emph{Neumann commutator trick}.  No
comparably simple invariant is available for general tuples of length at least $3$.  The same
commutator mechanism is central in recent work of Hirose and Monden, who prove
that the mapping class group of a closed orientable surface of genus $g\ge 8$
has infinitely many Nielsen classes of generating pairs \cite{HiroseMonden};
see also \cite{Monden}.

For linear groups, the Nielsen problem has been studied most explicitly in
rank one and over finite fields.  McCullough and Wanderley analyzed generating
pairs of $\SL(2,q)$ using the Higman invariant and commutator traces
\cite{McCulloughWanderley}, while Hirose and Monden proved that
$\SL(2,\Z)\cong\mathcal M_1$, the orientation-preserving mapping class group of the torus, has only finitely many Nielsen classes of
generating pairs \cite{HiroseMonden}.  In higher integral rank, Conder,
Liversidge and Vsemirnov give a recent account and new families of generating
pairs of $\SL(n,\Z)$ \cite{ConderLiversidgeVsemirnov}.  To the best of our
knowledge, it has not previously been explicitly observed that, for fixed
higher rank, $\GL(n,\Z)$ or $\SL(n,\Z)$ has infinitely many Nielsen classes of
generating pairs.  By comparison, Hirose and Monden's theorem implies the
corresponding infinitude for $\operatorname{Sp}(2g,\Z)$ when $g\ge 8$
\cite{HiroseMonden}.

Our first main result gives the integral general and special linear analogues.

\begin{mainA}
Let $n\ge 8$.  The following statements hold.
\begin{enumerate}[label=\textup{(\alph*)}]
\item The group $\GL(n,\Z)$ has countably infinitely many Nielsen equivalence
classes of generating pairs.  There exist pairwise Nielsen-inequivalent
generating pairs
\[
 (Z_-,H_m),\qquad m\ge1,
\]
such that
\[
 |H_m|=60,
 \qquad
 |Z_-|=
 \begin{cases}
 n,&n\text{ even},\\
 2n,&n\text{ odd}.
 \end{cases}
\]
\item The group $\SL(n,\Z)$ has countably infinitely many Nielsen equivalence
classes of generating pairs.  There exist pairwise Nielsen-inequivalent
generating pairs
\[
 (Z_+,H_m),\qquad m\ge1,
\]
such that
\[
 |H_m|=60,
 \qquad
 |Z_+|=
 \begin{cases}
 n,&n\text{ odd},\\
 2n,&n\text{ even}.
 \end{cases}
\]
\end{enumerate}
\end{mainA}

The matrices $Z_-$, $Z_+$, and $H_m$ occurring in Theorem~A are defined
explicitly in \cref{def:homology-matrices}.  Theorem~A is proved as
\cref{thm:linear-nielsen} by combining these explicit matrix families with the
rank-two commutator invariant and explicit trace calculations.

The automorphism groups of free groups also have classical small generating
sets.  B.~H. Neumann proved that $\Aut(F_n)$ is two-generated for $n\ge4$
\cite{BHN}; Newman treated ranks $2$ and $3$ \cite{MFN}.  Hence $\Out(F_n)$ is
two-generated for every $n\ge 2$, but the Nielsen classification of its
generating pairs appears to be much less understood.  Let
$\rho:\Out(F_n)\twoheadrightarrow\GL(n,\Z)$ be the homomorphism induced by
abelianization, and put
\[
 \SOut(F_n)=\rho^{-1}(\SL(n,\Z))
 =\ker\bigl(\det\circ\rho:\Out(F_n)\longrightarrow\{\pm1\}\bigr).
\]

\begin{mainB}
Let $n\ge 8$.  The following statements hold.
\begin{enumerate}[label=\textup{(\alph*)}]
\item The group $\Out(F_n)$ has countably infinitely many Nielsen equivalence
classes of generating pairs.  There exist pairwise Nielsen-inequivalent
generating pairs $(a_m,b_m)$, $m\ge1$, such that
\[
 |b_m|=60,
 \qquad
 |a_m|=
 \begin{cases}
 n,&n\text{ even},\\
 2n,&n\text{ odd}.
 \end{cases}
\]
Moreover, $a_m=a_{m'}$ for all $m,m'\ge 1$.
\item The group $\SOut(F_n)$ has countably infinitely many Nielsen equivalence
classes of generating pairs.  There exist pairwise Nielsen-inequivalent
generating pairs $(a_m^+,b_m^+)$, $m\ge1$, such that
\[
 |b_m^+|=60,
 \qquad
 |a_m^+|=
 \begin{cases}
 n,&n\text{ odd},\\
 2n,&n\text{ even}.
 \end{cases}
\]
Moreover, $a_m^+=a_{m'}^+$ for all $m,m'\ge 1$.
\end{enumerate}
\end{mainB}

The generating pairs in Theorem~B, together with the noncyclic homology
quotients $\GL(n,\Z)$ and $\SL(n,\Z)$, imply
\[
 d(\Out(F_n))=d(\SOut(F_n))=2
\]
for $n\ge 8$.  

We are not aware of a reference explicitly recording that $\SOut(F_n)$ is
two-generated.  For $n=1$ the group $\SOut(F_1)$ is trivial, while
$\SOut(F_2)\cong\SL(2,\Z)$ is classically two-generated.  For completeness,
in \cref{thm:special-two-generated} we record a short proof that
$\SAut(F_n)$ and $\SOut(F_n)$ are two-generated for every $n\ge 3$.
For $n\ge4$, the corresponding assertion for $\SAut(F_n)$ follows readily
from the classical Nielsen--Gersten generating description of
$\SAut(F_n)$~\cite{Gersten,SatohTwistedSecond}; we handle the case $n=3$ by
a separate explicit argument.  

Theorems~A and~B arise from one construction.  We place automorphisms of
orders $4$, $3$, and $5$ on disjoint free factors; coprimality recovers the
order-$4$ and order-$3$ factors by fixed powers, while a parameter is stored in
the order-$5$ factor.  Together with a signed cyclic basis permutation and
Gersten's Nielsen generators, these elements generate $\Aut(F_n)$ or
$\SAut(F_n)$.  On homology the construction gives the explicit matrix pairs
$(Z_-,H_m)$ and $(Z_+,H_m)$ introduced in \cref{def:homology-matrices}; their
commutator traces distinguish their Nielsen classes.  The linear inequivalence
then implies the corresponding outer inequivalence.  This argument proves Theorem~A as
\cref{thm:linear-nielsen} and Theorem~B as \cref{thm:main}.

Stabilization leads to a different problem.  For a finite tuple
$\mathbf g=(g_1,\ldots,g_r)$, put
\[
 \operatorname{st}(\mathbf g)=(g_1,\ldots,g_r,1).
\]
Kapovich and Weidmann constructed groups with generating tuples that remain
Nielsen inequivalent after stabilization \cite{KWsmall,KWrandom}; our families,
by contrast, collapse after one stabilization.

\begin{mainC}
Let $n\ge 8$, and let the pairs $(a_m,b_m)$ and $(a_m^+,b_m^+)$ be the
families from Theorem~B.  For all $m,\ell\ge1$,
\begin{enumerate}[label=\textup{(\alph*)}]
\item
\[
 (a_m,b_m,1)\sim_N(a_\ell,b_\ell,1)
 \quad\text{as triples in }\Out(F_n);
\]
\item
\[
 (a_m^+,b_m^+,1)\sim_N(a_\ell^+,b_\ell^+,1)
 \quad\text{as triples in }\SOut(F_n).
\]
\end{enumerate}
\end{mainC}

Theorem~C is proved in \cref{cor:families-stabilize}.  Any two generating
pairs of any group become Nielsen equivalent after two stabilizations, so the
first stabilization is the only potentially nontrivial one for pairs.  Thus the
particular families constructed here carry no one-stabilization obstruction.

Relation modules provide an algebraic obstruction at this level for other
groups and generating tuples.  If
$M_{\mathbf g}$ is the relation module of a generating tuple $\mathbf g$, then
\[
 M_{\operatorname{st}(\mathbf g)}\cong M_{\mathbf g}\oplus\Z G.
\]
Thus
\[
 M_{\mathbf g}\oplus\Z G\not\cong M_{\mathbf h}\oplus\Z G
\]
obstructs Nielsen equivalence after one stabilization.  Evans's
non-cancellation examples \cite{EvansRelationII} show that this obstruction is
genuine.  Harlander and Jensen's exotic relation module for $BS(2,3)$
\cite{HarlanderJensen} gives a rank-two example in which relation modules
distinguish the original pairs but the distinction disappears after one
stabilization.  Kapovich and Weidmann's two-generated small-cancellation
examples \cite{KWsmall} remain a natural unresolved test case.  We also record
a quotient-level refinement for future applications.

The final section lists open problems on low ranks, weakly redundant Nielsen
connectivity, and relation-module cancellation.

\section{Preliminaries}\label{sec:preliminaries}

\subsection{Free groups and their automorphisms}

\begin{conv}\label{conv:free-group-notation}
Let $n\ge 1$, let
\[
 F_n=F(x_1,\ldots,x_n),
\]
be the free group of rank $n$,  with $X=X_n=\{x_1,\ldots,x_n\}$ as the specified distinguished free basis.  We use the
following notation.
\begin{enumerate}[label=\textup{(\alph*)}]
\item Abelianization, with respect to the ordered basis
$([x_1],\ldots,[x_n])$ of $H_1(F_n;\Z)$, gives a surjective homomorphism
\[
 \rho:\Aut(F_n)\longrightarrow\GL(n,\Z).
\]
Since $\Inn(F_n)\le\ker\rho$, we use the same symbol for the induced
surjective homomorphism
\[
 \rho:\Out(F_n)\longrightarrow\GL(n,\Z).
\]
\item We put
\[
 \SAut(F_n)=\ker(\det\circ\rho:\Aut(F_n)\to\{\pm1\}),
 \qquad
 \SOut(F_n)=\SAut(F_n)/\Inn(F_n).
\]
Thus $\SOut(F_n)$ is naturally the index-two subgroup
$\ker(\det\circ\rho:\Out(F_n)\to\{\pm1\})$ of $\Out(F_n)$, and the
restriction $\rho:\SOut(F_n)\twoheadrightarrow\SL(n,\Z)$ is surjective.
\item For $\varphi\in\Aut(F_n)$, we write
\[
 \overline\varphi=\varphi\Inn(F_n)\in\Out(F_n).
\]
If $\varphi\in\SAut(F_n)$, then $\overline\varphi\in\SOut(F_n)$; we use
the same notation for this element of $\SOut(F_n)$.
\item The \emph{signed basis permutation subgroup} $W_n\le\Aut(F_n)$ consists
of the automorphisms $\chi$ for which there is a unique permutation
$\sigma\in S_n$ and there are unique signs
$\eps_1,\ldots,\eps_n\in\{\pm1\}$ such that
\[
 \chi(x_i)=x_{\sigma(i)}^{\eps_i}\qquad(1\le i\le n).
\]
Thus $W_n\cong(\Z/2\Z)^n\rtimes S_n$.  Put
\[
 W_n^+=W_n\cap\SAut(F_n).
\]
For $\chi\in W_n$, let $\eta(\chi)=\prod_{i=1}^n\eps_i$ denote the
sign-change parity of $\chi$.  Then
\[
 \det\rho(\chi)=\eta(\chi)\operatorname{sgn}(\sigma).
\]
\item Let $E_n\le W_n$ be the subgroup of pure sign changes with even parity;
explicitly, $\chi\in E_n$ if and only if $\sigma=\mathrm{id}$ and
$\eta(\chi)=1$.  The map $\chi\mapsto(\eta(\chi),\sigma)$ induces the identifications
\begin{equation}
 E_n\cong(\Z/2\Z)^{n-1},
 \qquad
 W_n/E_n\cong(\Z/2\Z)\times S_n,
 \qquad
 W_n^+/E_n\cong S_n.
\label{eq:Wquotients-short}
\end{equation}
\end{enumerate}
\end{conv}

\begin{conv}\label{conv:composition}
Throughout the paper, products of automorphisms are compositions with the
rightmost factor applied first:
\[
 (\varphi\psi)(g)=\varphi(\psi(g))\qquad(g\in F_n).
\]
For group commutators we use
\begin{equation}
 [g,h]=g^{-1}h^{-1}gh.
\label{eq:commutator-convention-short}
\end{equation}
For matrices acting on $H_1(F_n;\Z)$, the $j$-th column is the image of the
$j$-th standard basis vector.  For an element $g$ of finite order, $|g|$ denotes the order of $g$.
\end{conv}

\begin{conv}\label{conv:hom-epi}
Let $G$ be a group, let $n\ge 1$ and $F_n=F(X)=F(x_1,\dots, x_n)$.

\begin{enumerate}
\item We denote by $\operatorname{Hom}(F_n,G)$ the set of all homomorphisms
$\theta:F_n\to G$.
\item Using the distinguished basis $X=\{x_1,\dots,x_n\}$ of $F_n$, we identify
$G^n$ with $\operatorname{Hom}(F_n,G)$.  Namely, a tuple
$\mathbf g=(g_1,\dots,g_n)\in G^n$ is identified with the homomorphism
\begin{equation}
\theta_{\mathbf g}:F_n\to G,
\qquad
\theta_{\mathbf g}(x_i)=g_i\quad (i=1,\dots,n).
\label{eq:qtuple-short}
\end{equation}
Conversely, every $\theta:F_n\to G$ equals $\theta_{\mathbf g}$ for
$\mathbf g=(\theta(x_1),\dots,\theta(x_n))\in G^n$.
\item The group $\Aut(F_n)$ acts on the right on
$\operatorname{Hom}(F_n,G)$ by precomposition: for $\Phi\in\Aut(F_n)$ and
$\theta\in\operatorname{Hom}(F_n,G)$, put
$\theta\cdot\Phi:=\theta\circ\Phi$.
\item We denote by $\operatorname{Epi}(F_n,G)$ the (possibly empty) set of all
epimorphisms $\theta:F_n\to G$.  This subset of $\operatorname{Hom}(F_n,G)$ is invariant under the
$\Aut(F_n)$-action just defined.
\end{enumerate}

\end{conv}

\subsection{Nielsen equivalence and weakly redundant tuples}

\begin{defn}[Nielsen maps]\label{def:nielsen-maps}
Let $G$ be a group, let $n\ge 1$, and let $F_n=F(x_1,\dots,x_n)$.
\begin{enumerate}
\item For $1\le i\ne j\le n$ and $\eps\in\{\pm1\}$, define \emph{Nielsen maps}
$G^n\to G^n$ by
\begin{equation}
\begin{aligned}
 I_i(g_1,\ldots,g_n)
  &=(g_1,\ldots,g_{i-1},g_i^{-1},g_{i+1},\ldots,g_n),\\
 P_{ij}(g_1,\ldots,g_i,\ldots,g_j,\ldots,g_n)
  &=(g_1,\ldots,g_j,\ldots,g_i,\ldots,g_n),\\
 R_{ij}^{\eps}(g_1,\ldots,g_n)
  &=(g_1,\ldots,g_{i-1},g_i g_j^{\eps},g_{i+1},\ldots,g_n).
\end{aligned}
\label{eq:elementary-nielsen-maps-short}
\end{equation}
\item Taking $G=F_n$, define endomorphisms
$\iota_i,\pi_{ij},\tau_{ij,\eps}:F_n\to F_n$ by
\[
\begin{aligned}
(\iota_i(x_1),\dots,\iota_i(x_n))&=I_i(x_1,\dots,x_n),\\
(\pi_{ij}(x_1),\dots,\pi_{ij}(x_n))&=P_{ij}(x_1,\dots,x_n),\\
(\tau_{ij,\eps}(x_1),\dots,\tau_{ij,\eps}(x_n))&=R_{ij}^\eps(x_1,\dots,x_n).
\end{aligned}
\]
Let $\mathfrak N_n$ be the set of all these endomorphisms.
\item For $a,b\in\Xpm$ with $a\ne b^{\pm1}$, let $E_{ab}:F_n\to F_n$ be the
endomorphism determined by
\[
 E_{ab}(a)=ab,
 \qquad
 E_{ab}(c)=c\quad
 (c\in\Xpm\setminus\{a,a^{-1}\}).
\]
Let $\mathfrak E_n$ be the set of all such endomorphisms $E_{ab}$.
\end{enumerate}
\end{defn}

We record the following standard facts about these maps.
\begin{prop}\label{prop:gersten-generation}
The following statements hold.
\begin{enumerate}
\item For $n\ge1$, every element of $\mathfrak N_n\cup\mathfrak E_n$ is an
automorphism of $F_n$.  More precisely,
\[
 \iota_i^{-1}=\iota_i,\qquad
 \pi_{ij}^{-1}=\pi_{ij},\qquad
 \tau_{ij,\eps}^{-1}=\tau_{ij,-\eps},
\]
and, if $a=x_i$ and $b=x_j^\eps$, then
\[
 E_{ab}=\tau_{ij,\eps},
\]
whereas if $a=x_i^{-1}$ and $b=x_j^\eps$, then, with the composition convention
of \cref{conv:composition},
\[
 E_{ab}=\iota_i\circ\tau_{ij,\eps}\circ\iota_i.
\]
\item For $n\ge1$,
\[
\Aut(F_n)=\langle \mathfrak N_n\rangle.
\]
\item For $n\ge 3$,
\[
 \SAut(F_n)=\langle \mathfrak E_n\rangle.
\]
\end{enumerate}
\end{prop}
\begin{proof}
The inverse formulas in part~\textup{(1)} follow directly from the definitions,
and show that every element of $\mathfrak N_n$ is an automorphism.  The two
formulas for $E_{ab}$ follow by evaluating the displayed compositions on the
distinguished basis; hence the elements of $\mathfrak E_n$ are automorphisms as
well.

Part~\textup{(2)} goes back to the classical work of Nielsen; see
\cite[Ch.~5]{MKS}.  Part~\textup{(3)} is Gersten's generating theorem \cite{Gersten}; see also \cite[Theorem~2.1]{SatohTwistedSecond} for an explicit formulation for $n\ge 3$.
\end{proof}

In view of \cref{prop:gersten-generation}(1), the elements of $\mathfrak N_n$
are the \emph{Nielsen automorphisms} of $F_n$ used throughout the paper.

\begin{defn}\label{def:nielsen-equivalence}
Let $G$ be a group and let $n\ge1$.
\begin{enumerate}[label=\textup{(\alph*)}]
\item The maps $I_i$, $P_{ij}$, and $R_{ij}^{\eps}$ from
\eqref{eq:elementary-nielsen-maps-short} are the \emph{elementary Nielsen
transformations} on $G^n$.
\item Tuples $\mathbf g,\mathbf h\in G^n$ are \emph{Nielsen equivalent},
written $\mathbf g\sim_N\mathbf h$, if $\mathbf h$ is obtained from
$\mathbf g$ by a finite composition of elementary Nielsen transformations.
\item For $s\ge1$, the $s$-fold stabilization of $\mathbf g$ is
\begin{equation}
 \operatorname{st}_s(\mathbf g)
 =(g_1,\ldots,g_n,\underbrace{1,\ldots,1}_{s\text{ entries}}).
\label{eq:stabilization-short}
\end{equation}
We put $\operatorname{st}_0(\mathbf g)=\mathbf g$ and write
$\operatorname{st}(\mathbf g)=\operatorname{st}_1(\mathbf g)$.
\item A generating $n$-tuple $\mathbf g=(g_1,\ldots,g_n)$ for $G$ is
\emph{redundant} if a proper subtuple generates $G$.  A generating $n$-tuple is
\emph{weakly redundant} if it is Nielsen equivalent to a redundant tuple.
\end{enumerate}
\end{defn}

For an $n$-tuple $\mathbf g=(g_1,\dots, g_n)\in G^n$ we denote
\[
\langle \mathbf g\rangle=\langle g_1,\dots, g_n\rangle \le G.
\]

We record the following well-known basic facts regarding Nielsen transformations;
see \cite[Ch.~3]{MKS} for the relevant background.

\begin{prop}\label{prop:basic-nielsen-facts}
Let $G$ be a group and let $n\ge1$.  Then the following statements hold.
\begin{enumerate}[label=\textup{(\alph*)}]
\item The relation $\sim_N$ is an equivalence relation on $G^n$.
\item For $\mathbf g,\mathbf h\in G^n$, if $\mathbf g\sim_N\mathbf h$, then
\[
 \langle\mathbf g\rangle=\langle\mathbf h\rangle.
\]
\item For $\mathbf g,\mathbf h\in G^n$,
\begin{equation}
 \mathbf g\sim_N\mathbf h
 \quad\Longleftrightarrow\quad
 \theta_{\mathbf h}=\theta_{\mathbf g}\circ\Phi
 \text{ for some }\Phi\in\Aut(F_n).
\label{eq:nielsen-aut-short}
\end{equation}
\item If $\mathbf g$ is a generating $n$-tuple of $G$, then the following are
equivalent:
\begin{enumerate}[label=\textup{(\roman*)}]
\item $\mathbf g$ is weakly redundant;
\item $\mathbf g$ is Nielsen equivalent to a tuple having a trivial coordinate;
\item $\ker\theta_{\mathbf g}$ contains a primitive element of $F_n$ (that is,
an element belonging to some free basis of $F_n$).
\end{enumerate}
\end{enumerate}
\end{prop}

\begin{proof}
The inverse of $I_i$ is $I_i$, the inverse of $P_{ij}$ is $P_{ij}$, and the
inverse of $R_{ij}^{\eps}$ is $R_{ij}^{-\eps}$.  This proves part~\textup{(a)}.
Each elementary Nielsen transformation preserves the subgroup generated by the
coordinates, which proves part~\textup{(b)}.

Under the identification in \cref{conv:hom-epi}, the transformations
$I_i,P_{ij},R_{ij}^{\eps}$ correspond respectively to precomposition by
$\iota_i,\pi_{ij},\tau_{ij,\eps}$.  These automorphisms generate
$\Aut(F_n)$ by \cref{prop:gersten-generation}(2), proving part~\textup{(c)}.

For part~\textup{(d)}, a generating tuple with a trivial coordinate is
redundant.  Conversely, if a generating tuple is redundant, choose a coordinate
$g_i$ that belongs to the subgroup generated by the remaining coordinates.
Choose a word $W$ in the remaining coordinates such that
$W(g_1,\ldots,\widehat{g_i},\ldots,g_n)=g_i^{-1}$.  Applying to the $i$-th
coordinate the right Nielsen multiplications corresponding to the letters of
$W$ changes that coordinate to $1$ while fixing all other coordinates.  Thus \textup{(i)} and \textup{(ii)} are
equivalent.

If \textup{(ii)} holds, part~\textup{(c)} gives
$\theta_{\mathbf h}=\theta_{\mathbf g}\circ\Phi$ for some $\Phi\in\Aut(F_n)$
and some tuple $\mathbf h$ having, say, the $i$-th coordinate equal to $1$.
Then $\Phi(x_i)\in\ker\theta_{\mathbf g}$ and $\Phi(x_i)$ is primitive, so
\textup{(iii)} holds.  Conversely, if a primitive element
$p\in\ker\theta_{\mathbf g}$ is given, choose $\Phi\in\Aut(F_n)$ with
$\Phi(x_n)=p$.  The tuple corresponding to $\theta_{\mathbf g}\circ\Phi$ has
trivial last coordinate and is Nielsen equivalent to $\mathbf g$ by
part~\textup{(c)}.  Hence \textup{(iii)} implies \textup{(ii)}.
\end{proof}

\begin{defn}[Nielsen graphs]\label{def:nielsen-graph}
For $n\ge 1$, the \emph{Nielsen graph} $N_n(G)$ has the ordered generating
$n$-tuples of elements of $G$ as vertices and joins two vertices when one
elementary Nielsen transformation takes one tuple to the other.  The weakly
redundant vertices form a union of connected components; we denote the
resulting subgraph by $N_n^{\rm wr}(G)$.
\end{defn}
For background on Nielsen graphs and $\Aut(F_n)$-actions on generating tuples,
see Myropolska--Nagnibeda \cite{MyropolskaNagnibeda}.

\begin{prop}\label{prop:common-coordinate-stabilization}
Let $G$ be a group, let $n\ge1$, let $c_1,\ldots,c_{n-1},u,v\in G$, and
assume
\begin{equation}
 G=\langle c_1,\ldots,c_{n-1},u\rangle
 =\langle c_1,\ldots,c_{n-1},v\rangle.
\label{eq:common-generation-short}
\end{equation}
Then
\begin{equation}
 (c_1,\ldots,c_{n-1},u,1)
 \sim_N
 (c_1,\ldots,c_{n-1},v,1).
\label{eq:common-stabilization-short}
\end{equation}
\end{prop}

\begin{proof}
Choose
\[
 W=x_{i_1}^{\eps_1}\cdots x_{i_s}^{\eps_s}\in F_n
\]
with
\[
 W(c_1,\ldots,c_{n-1},u)=v.
\]
Starting from $(c_1,\ldots,c_{n-1},u,1)$, apply
\[
 R_{n+1,i_1}^{\eps_1},\ldots,R_{n+1,i_s}^{\eps_s}.
\]
The first $n$ coordinates remain fixed and the last coordinate becomes $v$.
Hence
\[
 (c_1,\ldots,c_{n-1},u,1)
 \sim_N
 (c_1,\ldots,c_{n-1},u,v).
\]
Apply $P_{n,n+1}$.  By \eqref{eq:common-generation-short}, choose
$W'\in F_n$ with
\[
 W'(c_1,\ldots,c_{n-1},v)=u^{-1}.
\]
Applying to the last coordinate the transformations corresponding to the
letters of $W'$ changes that coordinate from $u$ to $1$ and fixes the first
$n$ coordinates.  The resulting tuple is
$(c_1,\ldots,c_{n-1},v,1)$.
\end{proof}

The following classical proposition is the commutator, or Higman, invariant
for generating pairs; see, for example, \cite[Lemma~2.4]{Monden}.

\begin{prop}\label{prop:commutator-invariant}
Let $(g,h)$ and $(g',h')$ be Nielsen-equivalent ordered pairs in a group $G$.
Then $[g',h']$ is conjugate in $G$ to either $[g,h]$ or $[g,h]^{-1}$.
\end{prop}

\begin{proof}
The assertion follows by checking the elementary Nielsen transformations.
Interchanging the coordinates inverts the commutator.  Inverting the first
coordinate gives
\[
 [g^{-1},h]=g[g,h]^{-1}g^{-1}.
\]
The other inversion follows by symmetry.  Finally,
\[
 [gh,h]=h^{-1}[g,h]h,
 \qquad
 [gh^{-1},h]=h[g,h]h^{-1},
\]
and the transformations changing the second coordinate follow by symmetry.
\end{proof}

\section{Two-generation of the special automorphism groups}\label{sec:two-generation}

We record a subsidiary two-generation result.  The argument for $n\ge4$ is a
short consequence of Gersten's generating set from
\cref{prop:gersten-generation}, whereas rank $3$ requires a separate explicit
construction.

\begin{conv}\label{conv:two-generation-automorphisms}
Let $n\ge4$.  Define $\alpha,\beta\in\Aut(F_n)$ as follows.  The
automorphism $\alpha$ fixes every $x_i$ with $i\notin\{1,2\}$ and satisfies
\begin{equation}
 \alpha(x_1)=x_2,
 \qquad
 \alpha(x_2)=x_1^{-1}.
\label{eq:alpha-short}
\end{equation}
The automorphism $\beta$ fixes every $x_i$ with $i\notin\{3,4\}$ and
satisfies
\begin{equation}
 \beta(x_3)=x_4^{-1}x_3^{-1},
 \qquad
 \beta(x_4)=x_3.
\label{eq:beta-short}
\end{equation}
Also define the signed cyclic basis permutation $\zeta_+\in W_n$ by
\begin{equation}
 \zeta_+(x_i)=x_{i+1}\quad(1\le i<n),
 \qquad
 \zeta_+(x_n)=x_1^{(-1)^{n-1}}.
\label{eq:zeta-plus-short}
\end{equation}
\end{conv}

\begin{lem}\label{lem:transitivity-Wplus}
Let $n\ge 3$.  The group $W_n^+$ acts transitively on
\begin{equation}
 \mathcal A_n
 =\{(a,b)\in\Xpm\times\Xpm\mid a\ne b^{\pm1}\}.
\label{eq:admissible-pairs-short}
\end{equation}
\end{lem}

\begin{proof}
A signed basis permutation can send any pair in $\mathcal A_n$ to any other.
If the determinant of such a signed basis permutation is $-1$, invert one
basis element not used by the target pair.  Since $n\ge 3$, such a basis element
exists.  The additional inversion fixes the target pair and changes the
determinant.
\end{proof}

\begin{lem}\label{lem:signed-permutations-plus}
Let $n\ge4$, and let $\alpha,\zeta_+$ be as in
\cref{conv:two-generation-automorphisms}.  Then
\begin{equation}
 \langle \zeta_+,\alpha\rangle=W_n^+.
\label{eq:signed-generation-plus-short}
\end{equation}
\end{lem}

\begin{proof}
The underlying permutation of $\alpha$ is the transposition $(1\ 2)$ and the
sign-change parity of $\alpha$ is $-1$, so $\det\rho(\alpha)=1$.  The
underlying permutation of $\zeta_+$ is an $n$-cycle and the sign-change parity
of $\zeta_+$ is $(-1)^{n-1}$, so $\det\rho(\zeta_+)=1$.  Thus
$\alpha,\zeta_+\in W_n^+$.

The element $\alpha^2$ inverts $x_1,x_2$.  Conjugating $\alpha^2$ by powers of
$\zeta_+$ gives the adjacent double sign changes, which generate $E_n$.
Modulo $E_n$, the elements $\zeta_+$ and $\alpha$ induce an $n$-cycle and a
transposition in $W_n^+/E_n\cong S_n$.  Hence their images generate the
quotient, and therefore $\langle\zeta_+,\alpha\rangle=W_n^+$.
\end{proof}

\begin{thm}\label{thm:special-two-generated}
For every $n\ge 3$,
\[
 d(\SAut(F_n))=d(\SOut(F_n))=2.
\]
More precisely, for $n\ge 4$ the pair
\[
 (\zeta_+,\alpha\beta)
\]
from \cref{conv:two-generation-automorphisms} generates $\SAut(F_n)$.  For
$n=3$, the pair $(g,h)$ defined in \eqref{eq:rank-three-g-h} generates
$\SAut(F_3)$.
\end{thm}

\begin{proof}
Assume first that $n\ge4$ and put
\[
 \Omega=\alpha\beta.
\]
The defining formulas show that $\alpha,\beta,\zeta_+\in\SAut(F_n)$.  The
automorphisms $\alpha$ and $\beta$ have respective orders $4$ and $3$ and
have disjoint supports.  Hence they commute and
\begin{equation}
 \Omega^9=\alpha,
 \qquad
 \Omega^4=\beta.
\label{eq:recover-alpha-beta-12}
\end{equation}
By \cref{lem:signed-permutations-plus}, the group
$L=\langle\zeta_+,\Omega\rangle$ therefore contains $W_n^+$.

Let $\lambda,\nu\in\Aut(F_n)$ fix every basis element $x_i$ with
$i\notin\{3,4\}$ and satisfy
\[
 \lambda(x_3)=x_4^{-1},\qquad \lambda(x_4)=x_3,
 \qquad
 \nu(x_3)=x_3,\qquad \nu(x_4)=x_3x_4.
\]
Then $\lambda\in W_n^+$ and, by \eqref{eq:beta-short},
$\beta=\nu\lambda$.  Thus
\[
 \nu=\beta\lambda^{-1}\in L.
\]
Moreover,
\[
 \nu=E_{x_4^{-1},x_3^{-1}}.
\]
For $\chi\in W_n^+$ one has
\[
 \chi E_{ab}\chi^{-1}=E_{\chi(a),\chi(b)}.
\]
By \cref{lem:transitivity-Wplus}, the $W_n^+$-conjugates of $\nu$ contain
every element of $\mathfrak E_n$.  Gersten's generating theorem,
\cref{prop:gersten-generation}(3), now gives
\[
 \SAut(F_n)=\langle\zeta_+,\Omega\rangle
 \qquad(n\ge4).
\]

It remains to treat $n=3$.  Put
\[
 E=E_{x_1,x_2}\in\SAut(F_3),
\]
and define signed basis permutations $\sigma,\tau\in W_3$ by
\begin{equation}
\begin{aligned}
 \sigma(x_1)&=x_3^{-1},&
 \sigma(x_2)&=x_2,&
 \sigma(x_3)&=x_1,\\
 \tau(x_1)&=x_3^{-1},&
 \tau(x_2)&=x_1,&
 \tau(x_3)&=x_2^{-1}.
\end{aligned}
\label{eq:rank-three-sigma-tau}
\end{equation}
The sign-change parity and the permutation parity have the same sign for each
of $\sigma$ and $\tau$, so $\sigma,\tau\in W_3^+$.  The element
$\sigma^2$ inverts $x_1,x_3$.  Conjugating $\sigma^2$ by
$\tau$ and $\tau^2$ gives the other double sign changes, so
$E_3\le\langle\sigma,\tau\rangle$.  Modulo $E_3$, the automorphisms
$\sigma$ and $\tau$ induce, respectively, a transposition and a $3$-cycle in
$W_3^+/E_3\cong S_3$.  Consequently,
\begin{equation}
 \langle\sigma,\tau\rangle=W_3^+.
\label{eq:rank-three-Wplus}
\end{equation}

Now put
\begin{equation}
 g=\sigma E,
 \qquad
 h=\tau E.
\label{eq:rank-three-g-h}
\end{equation}
Since $\sigma,\tau,E\in\SAut(F_3)$, we have $g,h\in\SAut(F_3)$.  The
defining formulas give
\[
\begin{aligned}
 g(x_1)&=x_3^{-1}x_2,& g(x_2)&=x_2,& g(x_3)&=x_1,\\
 h(x_1)&=x_3^{-1}x_1,& h(x_2)&=x_1,& h(x_3)&=x_2^{-1}.
\end{aligned}
\]
Let $R$ denote the automorphism represented by the word
\[
 R=g h g h^{-1}g^{-1}h^{-1}g h^{-4}g^{-1}.
\]
For example, applying the four rightmost factors first and freely reducing
gives
\[
\begin{aligned}
 (h^{-1}gh^{-4}g^{-1})(x_1)&=x_3^{-1}x_2^2x_1^{-1},\\
 (h^{-1}gh^{-4}g^{-1})(x_2)&=x_1^{-1},\\
 (h^{-1}gh^{-4}g^{-1})(x_3)&=x_1^{-1}x_2^{-1}x_3.
\end{aligned}
\]
Continuing with the remaining factors gives
\[
 R(x_1)=x_1x_2,\qquad R(x_2)=x_2,\qquad R(x_3)=x_3.
\]
Consequently,
\begin{equation}
 E= g h g h^{-1}g^{-1}h^{-1}g h^{-4}g^{-1}.
\label{eq:rank-three-recover-E}
\end{equation}
Hence
$E\in\langle g,h\rangle$, and \eqref{eq:rank-three-g-h} gives
\[
 \sigma=gE^{-1},
 \qquad
 \tau=hE^{-1}.
\]
By \eqref{eq:rank-three-Wplus}, the subgroup $\langle g,h\rangle$ therefore
contains $W_3^+$.  Using \cref{lem:transitivity-Wplus} once more, the
$W_3^+$-conjugates of $E=E_{x_1,x_2}$ contain every element of
$\mathfrak E_3$.  Thus \cref{prop:gersten-generation}(3) yields
\[
 \SAut(F_3)=\langle g,h\rangle.
\]

For every $n\ge 3$, passage to the quotient by $\Inn(F_n)$ shows that
$\SOut(F_n)$ is also generated by two elements.  Finally, both
$\SAut(F_n)$ and $\SOut(F_n)$ map onto the noncyclic group $\SL(n,\Z)$ under
$\rho$, so neither group is cyclic.  Hence both groups have minimal number of
generators equal to $2$.
\end{proof}

\section{Explicit torsion generating pairs}\label{sec:generation}

Assume throughout this section that $n\ge 8$.

We retain the automorphisms $\alpha,\beta,\zeta_+$ from
\cref{conv:two-generation-automorphisms}.

\begin{conv}\label{conv:torsion-automorphisms}
The following additional notation will be used.
\begin{enumerate}[label=\textup{(\alph*)}]
\item Let $\gamma\in\Aut(F(x_5,x_6,x_7,x_8))$ be given by
\begin{equation}
 \gamma(x_5)=x_6,
 \quad \gamma(x_6)=x_7,
 \quad \gamma(x_7)=x_8,
 \quad \gamma(x_8)=(x_5x_6x_7x_8)^{-1}.
\label{eq:gamma-short}
\end{equation}
This assignment defines an automorphism: the inverse of $\gamma$ sends
$x_5$ to $(x_5x_6x_7x_8)^{-1}$ and sends $x_6,x_7,x_8$ to
$x_5,x_6,x_7$, respectively.  Extend $\gamma$ to $F_n$, again denoted by
$\gamma\in\Aut(F_n)$, by setting $\gamma(x_i)=x_i$ for
$i\notin\{5,6,7,8\}$.
\item Let $\mu\in\Aut(F(x_5,x_6))$ be given by
\begin{equation}
 \mu(x_5)=x_5,
 \qquad
 \mu(x_6)=x_5x_6.
\label{eq:mu-short}
\end{equation}
Extend $\mu$ to $F_n$, again denoted by $\mu\in\Aut(F_n)$, by setting
$\mu(x_i)=x_i$ for $i\notin\{5,6\}$.  For $m\ge1$,
put, in $\Aut(F_n)$,
\begin{equation}
 \gamma_m=\mu^m\gamma\mu^{-m},
 \qquad
 \omega_m=\alpha\beta\gamma_m.
\label{eq:gamma-omega-short}
\end{equation}
\end{enumerate}
\end{conv}

\begin{lem}\label{lem:torsion-packing}
Let $n\ge 8$ and $m\ge1$.  In the group $\Aut(F_n)$, the following statements
hold.
\begin{enumerate}[label=\textup{(\alph*)}]
\item The elements $\alpha,\beta,\gamma_m$ have respective orders $4,3,5$ and commute
pairwise.
\item The element $\omega_m$ belongs to $\SAut(F_n)$ and has order $60$ as an
element of $\Aut(F_n)$.
\item The identities
\begin{equation}
 \omega_m^{45}=\alpha,
 \qquad
 \omega_m^{40}=\beta,
 \qquad
 \omega_m^{36}=\gamma_m
\label{eq:recover-factors-short}
\end{equation}
hold in $\Aut(F_n)$.
\end{enumerate}
\end{lem}

\begin{proof}
Equation \eqref{eq:alpha-short} gives
$\alpha^2(x_1)=x_1^{-1}$ and $\alpha^2(x_2)=x_2^{-1}$, so $|\alpha|=4$.  From
\eqref{eq:beta-short},
\[
 \beta^2(x_3)=x_4,
 \qquad
 \beta^2(x_4)=x_4^{-1}x_3^{-1},
\]
so $|\beta|=3$.  If
\[
 y=(x_5x_6x_7x_8)^{-1},
\]
then $\gamma(y)=x_5$, and hence $\gamma$ cyclically permutes
$x_5,x_6,x_7,x_8,y$.  Thus $|\gamma|=|\gamma_m|=5$.
The three supports are disjoint, so the automorphisms commute.  Their
homological determinants are $1$, and hence $\omega_m\in\SAut(F_n)$.  Since the
orders $4,3,5$ are pairwise coprime, $|\omega_m|=60$.  Reducing the exponents
$45,40,36$ modulo $4,3,5$ gives \eqref{eq:recover-factors-short}.
\end{proof}

For $\delta\in\{+1,-1\}$, put
\begin{equation}
 \eps_\delta=\delta(-1)^{n-1}
\label{eq:epsdelta-short}
\end{equation}
and define the signed cyclic permutation $\zeta_\delta\in W_n$ by
\begin{equation}
 \zeta_\delta(x_i)=x_{i+1}\quad(1\le i<n),
 \qquad
 \zeta_\delta(x_n)=x_1^{\eps_\delta}.
\label{eq:zeta-short}
\end{equation}
For $\delta=+1$, this agrees with the automorphism $\zeta_+$ from
\cref{conv:two-generation-automorphisms}.  Then
\begin{equation}
 \det\rho(\zeta_\delta)=\delta.
\label{eq:zeta-det-short}
\end{equation}

\begin{lem}\label{lem:signed-permutations-minus}
Let $n\ge 8$.  Then
\begin{equation}
 \langle \zeta_-,\alpha\rangle=W_n.
\label{eq:signed-generation-minus-short}
\end{equation}
\end{lem}

\begin{proof}
The element $\alpha^2$ inverts $x_1,x_2$.  Conjugating $\alpha^2$ by powers
of $\zeta_-$ gives the adjacent double sign changes, which generate $E_n$.
Put $L=\langle\zeta_-,\alpha\rangle$.  The image of $L/E_n$ projects onto
$S_n$ in $(\Z/2\Z)\times S_n$.  If it has nontrivial intersection with the
central $\Z/2\Z$, then it contains that factor and therefore equals
$(\Z/2\Z)\times S_n$.  Thus a proper subgroup with this projection must have
trivial intersection with the central factor and hence must be the graph of a
homomorphism $S_n\to\Z/2\Z$, necessarily either the trivial map or the sign
map.  The image of $\alpha$ excludes the trivial graph, while
\eqref{eq:zeta-det-short} shows that the sign-change parity of $\zeta_-$
differs from the permutation parity of the underlying $n$-cycle of $\zeta_-$,
excluding the sign graph.  Hence
$L/E_n=(\Z/2\Z)\times S_n$ and $L=W_n$.
\end{proof}

\begin{prop}\label{prop:generation-aut-saut}
Let $n\ge 8$ and $m\ge1$.  Then
\begin{enumerate}[label=\textup{(\alph*)}]
\item $\langle \zeta_+,\omega_m\rangle=\SAut(F_n)$;
\item $\langle \zeta_-,\omega_m\rangle=\Aut(F_n)$.
\end{enumerate}
\end{prop}

\begin{proof}
By \cref{lem:torsion-packing}, each group
$\langle\zeta_\delta,\omega_m\rangle$ contains
$\alpha=\omega_m^{45}$ and $\beta=\omega_m^{40}$, and hence contains
$\alpha\beta$.  Therefore \cref{thm:special-two-generated} gives
\[
 \langle\zeta_+,\omega_m\rangle=\SAut(F_n).
\]

For the negative-sign case, \cref{lem:signed-permutations-minus} gives
$W_n\le\langle\zeta_-,\omega_m\rangle$.  In particular this group contains
$\zeta_+$, as well as $\alpha\beta$, so
\cref{thm:special-two-generated} implies that it contains $\SAut(F_n)$.
It also contains the determinant-$-1$ element $\zeta_-$.  Since
$\SAut(F_n)$ has index two in $\Aut(F_n)$, it follows that
$\langle\zeta_-,\omega_m\rangle=\Aut(F_n)$.
\end{proof}

\begin{cor}\label{cor:generation-out-sout}
Let $n\ge 8$ and $m\ge1$.  The following statements hold.
\begin{enumerate}[label=\textup{(\alph*)}]
\item The pair $(\overline{\zeta}_+,\overline{\omega}_m)$ generates $\SOut(F_n)$, and
the order of $\overline{\zeta}_+$ in $\SOut(F_n)$ is
\[
 |\overline{\zeta}_+|=
 \begin{cases}
 n,&n\text{ odd},\\
 2n,&n\text{ even};
 \end{cases}
\]
\item The pair $(\overline{\zeta}_-,\overline{\omega}_m)$ generates $\Out(F_n)$, and
the order of $\overline{\zeta}_-$ in $\Out(F_n)$ is
\[
 |\overline{\zeta}_-|=
 \begin{cases}
 2n,&n\text{ odd},\\
 n,&n\text{ even};
 \end{cases}
\]
\item The element $\overline{\omega}_m$ has order $60$ in $\Out(F_n)$ and,
equivalently, in the subgroup $\SOut(F_n)$.
\end{enumerate}
\end{cor}

\begin{proof}
The generation statements follow from \cref{prop:generation-aut-saut}.
By \cref{lem:torsion-packing}, $\omega_m$ has order $60$ in $\Aut(F_n)$,
and by \cref{lem:homology-matrices}(b), the homology matrix
$H_m=\rho(\omega_m)$ has order $60$.  Therefore the image
$\overline{\omega}_m$ has order $60$ in $\Out(F_n)$ and hence
also in $\SOut(F_n)$.  Also,
\[
 \zeta_\delta^n(x_i)=x_i^{\eps_\delta}
\]
for every $i$.  Hence $\zeta_\delta$ has order $n$ when $\eps_\delta=1$ and
order $2n$ when $\eps_\delta=-1$.  The homology matrix of $\zeta_\delta$
has the same order, so passage to the relevant outer quotient does not shorten
this order.  Substituting
$\eps_+=(-1)^{n-1}$ and $\eps_-=(-1)^n$ gives the stated orders in
$\SOut(F_n)$ and $\Out(F_n)$, respectively.
\end{proof}

\section{Commutator traces and the main theorems}\label{sec:main}

Assume throughout this section that $n\ge 8$, and put
$e_i=[x_i]\in H_1(F_n;\Z)$.

\begin{defn}[Matrix notation for the homology construction]
\label{def:homology-matrices}
We use the following notation relative to the ordered basis
$(e_1,\ldots,e_n)$.
\begin{enumerate}[label=\textup{(\alph*)}]
\item On the coordinate pairs $(e_1,e_2)$ and $(e_3,e_4)$, respectively, put
\begin{equation}
 A=
 \begin{pmatrix}0&-1\\1&0\end{pmatrix},
 \qquad
 B=
 \begin{pmatrix}-1&1\\-1&0\end{pmatrix}.
\label{eq:AB-short}
\end{equation}
\item For $m\ge1$, on $\langle e_5,e_6,e_7,e_8\rangle$, put
\begin{equation}
 U_m=
 \begin{pmatrix}
 1&m&0&0\\0&1&0&0\\0&0&1&0\\0&0&0&1
 \end{pmatrix},
 \qquad
 P=
 \begin{pmatrix}
 0&0&0&-1\\1&0&0&-1\\0&1&0&-1\\0&0&1&-1
 \end{pmatrix},
\label{eq:UP-short}
\end{equation}
and define
\begin{equation}
 D_m=U_mPU_m^{-1}
 =
 \begin{pmatrix}
 m&-m^2&0&-m-1\\
 1&-m&0&-1\\
 0&1&0&-1\\
 0&0&1&-1
 \end{pmatrix}.
\label{eq:Dm-short}
\end{equation}
\item Define the $n\times n$ matrix
\begin{equation}
 H_m=\diag(A,B,D_m,I_{n-8}).
\label{eq:Hm-short}
\end{equation}
\item For $\delta\in\{+1,-1\}$, with
$\eps_\delta=\delta(-1)^{n-1}$ as in \eqref{eq:epsdelta-short}, define the
$n\times n$ signed cyclic permutation matrix
\begin{equation}
 Z_\delta=
 \begin{pmatrix}
 0&0&\cdots&0&\eps_\delta\\
 1&0&\cdots&0&0\\
 0&1&\ddots&0&0\\
 \vdots&&\ddots&\ddots&\vdots\\
 0&0&\cdots&1&0
 \end{pmatrix}.
\label{eq:Zdelta-short}
\end{equation}
\item For $m\ge1$ and $\delta\in\{+1,-1\}$, define
\begin{equation}
 K_{\delta,m}
 =Z_\delta^{-1}H_m^{-1}Z_\delta H_m.
\label{eq:Kdelta-short}
\end{equation}
\end{enumerate}
\end{defn}

\begin{lem}\label{lem:homology-matrices}
With the notation of \cref{def:homology-matrices}, the following statements
hold.
\begin{enumerate}[label=\textup{(\alph*)}]
\item The matrices of \cref{def:homology-matrices} are related to the automorphisms from
\cref{conv:torsion-automorphisms} by
\[
 \rho(\alpha)=\diag(A,I_{n-2}),
 \qquad
 \rho(\beta)=\diag(I_2,B,I_{n-4}),
\]
\[
 \rho(\mu^m)=\diag(I_4,U_m,I_{n-8}),
 \qquad
 \rho(\gamma)=\diag(I_4,P,I_{n-8}),
\]
\[
 \rho(\gamma_m)=\diag(I_4,D_m,I_{n-8}),
 \qquad
 \rho(\omega_m)=H_m,
 \qquad
 \rho(\zeta_\delta)=Z_\delta.
\]
In particular, $K_{\delta,m}$ is the homology matrix of
$[\overline{\zeta}_\delta,\overline{\omega}_m]$.
\item The matrix $H_m$ has order $60$.  Moreover,
\[
 \det Z_\delta=\delta,
 \qquad
 |Z_\delta|=
 \begin{cases}
 n,&\eps_\delta=1,\\
 2n,&\eps_\delta=-1.
 \end{cases}
\]
\end{enumerate}
\end{lem}

\begin{proof}
Part~\textup{(a)} follows directly from
\cref{conv:torsion-automorphisms}, the definition of $\zeta_\delta$ in
\eqref{eq:zeta-short}, and the column convention in \cref{conv:composition}.
For part~\textup{(b)}, the matrices $A$, $B$, and $P$ have respective orders
$4$, $3$, and $5$, and $D_m$ is conjugate to $P$.  Hence the block diagonal
matrix $H_m$ has order $\operatorname{lcm}(4,3,5)=60$.
The matrix $Z_\delta$ has the underlying permutation of an $n$-cycle and
satisfies $Z_\delta^n=\eps_\delta I_n$.  This gives the asserted order, while
\eqref{eq:zeta-det-short} and part~\textup{(a)} give
$\det Z_\delta=\delta$.
\end{proof}

\begin{prop}\label{prop:trace-formulas}
Let $n\ge 8$, $m\ge1$, and $\delta\in\{+1,-1\}$.  The following statements hold.
\begin{enumerate}[label=\textup{(\alph*)}]
\item if $n=8$, then
\begin{equation}
 \tr(K_{\delta,m})=-m^3+m^2+3,
 \qquad
 \tr(K_{\delta,m}^{-1})=2m^2+3;
\label{eq:trace-eight-short}
\end{equation}
\item if $n\ge9$, then
\begin{equation}
 \tr(K_{\delta,m})=-m^3+m^2+n-7,
 \qquad
 \tr(K_{\delta,m}^{-1})=2m^2+n-6.
\label{eq:trace-large-short}
\end{equation}
\end{enumerate}
\end{prop}

\begin{proof}
Since $H_m$ is block diagonal and $Z_\delta$ is a signed cyclic shift, the
multiplication in \eqref{eq:Kdelta-short} can be carried out one coordinate at
a time.  For example, when $n\ge9$,
\[
 H_m e_5=me_5+e_6,
 \qquad
 Z_\delta H_m e_5=me_6+e_7,
\]
and a direct calculation in the $D_m$ block gives
\[
 H_m^{-1}(me_6+e_7)
 =m(m^2+m+2)e_5+(m^2+1)e_6+m^2e_7+m^2e_8.
\]
Applying $Z_\delta^{-1}$ therefore shows that the fifth diagonal entry of
$K_{\delta,m}$ is $m^2+1$.  The other entries are obtained in the same way.
For $n\ge9$, the first eight diagonal entries are as follows:
\begin{equation}
\begin{array}{c|cccccccc}
 i&1&2&3&4&5&6&7&8\\ \hline
 (K_{\delta,m})_{ii}
   &0&0&1&0&m^2+1&1-m^3&0&-1\\
 (K_{\delta,m}^{-1})_{ii}
   &0&0&0&-m&2m^2+2m+1&1&1-m&0.
\end{array}
\label{eq:diagonal-large-short}
\end{equation}
For $9\le i<n$, both diagonal entries are $1$, while the $n$-th diagonal
entry of each matrix is $0$; the first range is empty when $n=9$.  Summing
gives \eqref{eq:trace-large-short}.

For $n=8$, the identity tail is absent and the cyclic shift wraps directly
from the $D_m$ block to the $A$ block.  Direct multiplication gives
\begin{equation}
\begin{array}{c|cccccccc}
 i&1&2&3&4&5&6&7&8\\ \hline
 (K_{\delta,m})_{ii}
   &0&0&1&0&m^2+1&1-m^3&0&0\\
 (K_{\delta,m}^{-1})_{ii}
   &0&0&0&-m&2m^2+2m+1&1&1-m&0.
\end{array}
\label{eq:diagonal-eight-short}
\end{equation}
This yields \eqref{eq:trace-eight-short}.  Whenever the cyclic shift wraps
around, the sign $\eps_\delta$ occurs once from $Z_\delta$ and once from
$Z_\delta^{-1}$, so it cancels from the diagonal coefficients.
\end{proof}

\begin{cor}\label{cor:commutators-separated}
Let $n\ge 8$, let $m,\ell\ge1$ with $m\ne\ell$, and let
$\delta\in\{+1,-1\}$.  Then $K_{\delta,m}$ is conjugate in
$\GL(n,\Z)$ to neither $K_{\delta,\ell}$ nor $K_{\delta,\ell}^{-1}$.
\end{cor}

\begin{proof}
The function
\[
 f(m)=-m^3+m^2
\]
is strictly decreasing on the positive integers because
\[
 f(m+1)-f(m)=-3m^2-m<0.
\]
Hence $\tr(K_{\delta,m})\ne\tr(K_{\delta,\ell})$ for $m\ne\ell$.
Furthermore, if $n\ge9$,
\[
 \tr(K_{\delta,m})\le n-7<n-4\le\tr(K_{\delta,\ell}^{-1}),
\]
while for $n=8$,
\[
 \tr(K_{\delta,m})\le3<5\le\tr(K_{\delta,\ell}^{-1}).
\]
Trace invariance under conjugacy gives the conclusion.
\end{proof}

For $m\ge1$, put
\begin{equation}
 (a_m,b_m)=(\overline{\zeta}_-,\overline{\omega}_m)\in\Out(F_n)^2,
 \qquad
 (a_m^+,b_m^+)=(\overline{\zeta}_+,\overline{\omega}_m)\in\SOut(F_n)^2.
\label{eq:outer-families-short}
\end{equation}
Thus
\begin{equation}
 (\rho(a_m),\rho(b_m))=(Z_-,H_m),
 \qquad
 (\rho(a_m^+),\rho(b_m^+))=(Z_+,H_m).
\label{eq:linear-families-short}
\end{equation}

\Cref{cor:commutators-separated} gives the corresponding Nielsen inequivalence
in the linear quotients and leads to Theorem~A.

\begin{thm}\label{thm:linear-nielsen}
Let $n\ge 8$.  The following statements hold.
\begin{enumerate}[label=\textup{(\alph*)}]
\item The pairs
\[
 (Z_-,H_m)=(\rho(a_m),\rho(b_m)),
 \qquad m\ge1,
\]
are generating pairs of $\GL(n,\Z)$ and represent pairwise distinct Nielsen
equivalence classes.  They satisfy
\[
 |H_m|=60,
 \qquad
 |Z_-|=
 \begin{cases}
 n,&n\text{ even},\\
 2n,&n\text{ odd}.
 \end{cases}
\]
\item The pairs
\[
 (Z_+,H_m)=(\rho(a_m^+),\rho(b_m^+)),
 \qquad m\ge1,
\]
are generating pairs of $\SL(n,\Z)$ and represent pairwise distinct Nielsen
equivalence classes.  They satisfy
\[
 |H_m|=60,
 \qquad
 |Z_+|=
 \begin{cases}
 n,&n\text{ odd},\\
 2n,&n\text{ even}.
 \end{cases}
\]
\end{enumerate}
In particular, both $\GL(n,\Z)$ and $\SL(n,\Z)$ have countably infinitely
many Nielsen equivalence classes of generating pairs.
\end{thm}

\begin{proof}
By \cref{cor:generation-out-sout} and the surjectivity of the homology maps in
\cref{conv:free-group-notation}, every pair $(Z_-,H_m)$ generates
$\GL(n,\Z)$ and every pair $(Z_+,H_m)$ generates $\SL(n,\Z)$.  The order
assertions follow from \cref{lem:homology-matrices} and
$\eps_-=(-1)^n$, $\eps_+=(-1)^{n-1}$.

For part~\textup{(a)}, let $m\ne\ell$.  If $(Z_-,H_m)$ and $(Z_-,H_\ell)$
were Nielsen equivalent in $\GL(n,\Z)$, then
\cref{prop:commutator-invariant} would imply that their commutators
$K_{-,m}$ and $K_{-,\ell}$ are conjugate in $\GL(n,\Z)$ up to inversion.
This contradicts \cref{cor:commutators-separated}.  Hence the displayed pairs
represent pairwise distinct Nielsen equivalence classes in $\GL(n,\Z)$.

For part~\textup{(b)}, suppose that $(Z_+,H_m)$ and $(Z_+,H_\ell)$ were
Nielsen equivalent in $\SL(n,\Z)$ for some $m\ne\ell$.  Their commutators
$K_{+,m}$ and $K_{+,\ell}$ would then be conjugate in $\SL(n,\Z)$ up to
inversion.  Such a conjugacy is, in particular, a conjugacy in $\GL(n,\Z)$.
Equivalently, for elements of $\SL(n,\Z)$, non-conjugacy in $\GL(n,\Z)$
implies non-conjugacy in $\SL(n,\Z)$.  This again contradicts
\cref{cor:commutators-separated}.  Thus the displayed pairs represent
pairwise distinct Nielsen equivalence classes in $\SL(n,\Z)$.  Since
$\GL(n,\Z)$ and $\SL(n,\Z)$ are countable, the total number of Nielsen
equivalence classes of generating pairs in each group is countably infinite.
\end{proof}

The following theorem proves Theorem~B.

\begin{thm}\label{thm:main}
Let $n\ge 8$.  The following statements hold.
\begin{enumerate}[label=\textup{(\alph*)}]
\item $\Out(F_n)$ has countably infinitely many Nielsen equivalence classes of
generating pairs.  The pairs $(a_m,b_m)$ from
\eqref{eq:outer-families-short} represent pairwise distinct classes and satisfy
\[
 |b_m|=60,
 \qquad
 |a_m|=
 \begin{cases}
 n,&n\text{ even},\\
 2n,&n\text{ odd};
 \end{cases}
\]
Moreover, $a_m=a_\ell$ for all $m,\ell\ge1$.
\item $\SOut(F_n)$ has countably infinitely many Nielsen equivalence classes
of generating pairs.  The pairs $(a_m^+,b_m^+)$ from
\eqref{eq:outer-families-short} represent pairwise distinct classes and satisfy
\[
 |b_m^+|=60,
 \qquad
 |a_m^+|=
 \begin{cases}
 n,&n\text{ odd},\\
 2n,&n\text{ even};
 \end{cases}
\]
Moreover, $a_m^+=a_\ell^+$ for all $m,\ell\ge1$.
\item the minimal number of generators of each of $\Out(F_n)$ and
$\SOut(F_n)$ is $2$.
\end{enumerate}
\end{thm}

\begin{proof}
For part~\textup{(a)}, \cref{cor:generation-out-sout}(b),(c) shows that every
pair $(a_m,b_m)$ generates $\Out(F_n)$ and has the stated orders.  The
equality $a_m=a_\ell$ follows immediately from
\eqref{eq:outer-families-short}.  By
\cref{thm:linear-nielsen}(a), the image pairs
$(\rho(a_m),\rho(b_m))=(Z_-,H_m)$ represent pairwise distinct Nielsen
equivalence classes in $\GL(n,\Z)$.  A group homomorphism sends Nielsen
equivalent tuples to Nielsen equivalent tuples.  Therefore two distinct pairs
$(a_m,b_m)$ and $(a_\ell,b_\ell)$ cannot be Nielsen equivalent in
$\Out(F_n)$, since otherwise their homology images would be Nielsen equivalent
in $\GL(n,\Z)$.  Thus the displayed family gives infinitely many distinct
Nielsen classes in $\Out(F_n)$.  Since $\Out(F_n)$ is countable, the total
number of Nielsen equivalence classes of generating pairs is countably
infinite.

For part~\textup{(b)}, \cref{cor:generation-out-sout}(a),(c) gives the stated
generation and order assertions for $(a_m^+,b_m^+)$ in $\SOut(F_n)$.  The
equality $a_m^+=a_\ell^+$ follows immediately from
\eqref{eq:outer-families-short}.  By
\cref{thm:linear-nielsen}(b), their images
$(\rho(a_m^+),\rho(b_m^+))=(Z_+,H_m)$ represent pairwise distinct Nielsen
equivalence classes in $\SL(n,\Z)$.  Hence the same homomorphism argument
shows that the pairs $(a_m^+,b_m^+)$ represent pairwise distinct Nielsen
classes in $\SOut(F_n)$.  Since $\SOut(F_n)$ is countable, it has countably
infinitely many Nielsen equivalence classes of generating pairs.

Finally, for part~\textup{(c)}, the displayed generating pairs show that each
of $\Out(F_n)$ and $\SOut(F_n)$ can be generated by two elements.  On the
other hand, the surjective homology maps from
\cref{conv:free-group-notation} have the noncyclic images $\GL(n,\Z)$ and
$\SL(n,\Z)$, respectively.  Hence neither outer automorphism group is cyclic,
and the minimal number of generators of each is exactly $2$.
\end{proof}

\section{Stabilization and weakly redundant Nielsen equivalence}\label{sec:stabilization}

The first coordinates of the two families in \cref{thm:main} do not depend on
$m$.  The common-coordinate lemma therefore gives the following statement,
which proves Theorem~C.

\begin{cor}\label{cor:families-stabilize}
Let $n\ge 8$ and $m,\ell\ge1$.  The following statements hold.
\begin{enumerate}[label=\textup{(\alph*)}]
\item
\begin{equation}
 (\overline{\zeta}_-,\overline{\omega}_m,1)
 \sim_N
 (\overline{\zeta}_-,\overline{\omega}_\ell,1)
 \quad\text{as triples in }\Out(F_n);
\label{eq:out-stabilizes-short}
\end{equation}
\item
\begin{equation}
 (\overline{\zeta}_+,\overline{\omega}_m,1)
 \sim_N
 (\overline{\zeta}_+,\overline{\omega}_\ell,1)
 \quad\text{as triples in }\SOut(F_n).
\label{eq:sout-stabilizes-short}
\end{equation}
\end{enumerate}
\end{cor}

\begin{proof}
Apply \cref{prop:common-coordinate-stabilization} with the fixed first
coordinate $\overline{\zeta}_-$ or $\overline{\zeta}_+$ and use
\cref{cor:generation-out-sout}.
\end{proof}

The next elementary observation shows why one stabilization is the only stable
question for generating pairs.

\begin{prop}\label{prop:two-stabilizations}
Let $G$ be a group and let $(a,b)$ and $(c,d)$ be ordered generating pairs of
$G$.
Then
\begin{equation}
 (a,b,1,1)\sim_N(c,d,1,1).
\label{eq:two-stabilizations-short}
\end{equation}
\end{prop}

\begin{proof}
Choose words $C,D\in F(y_1,y_2)$ satisfying
\[
 C(a,b)=c,
 \qquad
 D(a,b)=d.
\]
Using the transformations $R_{3,j}^{\eps}$ associated to the letters of $C$
and the transformations $R_{4,j}^{\eps}$ associated to the letters of $D$,
change $(a,b,1,1)$ to $(a,b,c,d)$.  Permute the coordinates to obtain
$(c,d,a,b)$.  Choose $A,B\in F(y_1,y_2)$ with
\[
 A(c,d)=a^{-1},
 \qquad
 B(c,d)=b^{-1}.
\]
Apply the corresponding elementary Nielsen transformations to the third and
fourth coordinates.  They become $1$, while the first two coordinates remain
$c,d$.  The resulting tuple is $(c,d,1,1)$.
\end{proof}

For a group $G$, let $N_3^{\rm wr}(G)$ be the weakly redundant Nielsen
subgraph from \cref{def:nielsen-graph}.  The following question is therefore the
natural stable counterpart of the rank-two Nielsen problem:
\begin{equation}
 \text{is }N_3^{\rm wr}(\Out(F_n))\text{ connected?}
\label{eq:redundant-connectivity-question}
\end{equation}
The same question applies to $\SOut(F_n)$.  A positive answer would imply that
all generating pairs become Nielsen equivalent after one stabilization,
whereas a negative answer would produce a genuinely stable Nielsen invariant.

\section{Relation modules and one-stabilization obstructions}\label{sec:relation-modules}

Let $G$ be a group, put $\Lambda=\Z G$, and let
$I_G=\ker(\varepsilon:\Lambda\to\Z)$ be the augmentation ideal.  We use left
$\Lambda$-modules.

\begin{defn}\label{def:relation-module}
Let $G$ be a group, let $n\ge 1$, and let
$\mathbf g=(g_1,\ldots,g_n)$ be an ordered generating $n$-tuple of $G$,
and let
\[
 R_{\mathbf g}=\ker(\theta_{\mathbf g}:F_n\twoheadrightarrow G).
\]
The \emph{relation module associated to $\mathbf g$} is the abelian group
\begin{equation}
 M_{\mathbf g}=R_{\mathbf g}/[R_{\mathbf g},R_{\mathbf g}],
\label{eq:relation-module-short}
\end{equation}
equipped with the action specified by
\[
 g\cdot[r]=[\widetilde g r\widetilde g^{-1}],
\]
where $\widetilde g\in F_n$ is a lift of $g\in G$.  The well-definedness of
this action is recorded in \cref{prop:relation-module-properties}(a).
\end{defn}

Relation modules associated to free presentations are classical; see, for
example, Evans \cite{EvansRelation}.  The following proposition collects the
properties relevant to stabilization.

\begin{prop}\label{prop:relation-module-properties}
Let $G$ be a group and $\Lambda=\Z G$.
\begin{enumerate}[label=\textup{(\alph*)}]
\item Let $n\ge1$, and let $\mathbf g$ be an ordered generating $n$-tuple of $G$.
The formula in \cref{def:relation-module} defines a well-defined left
$G$-action on $M_{\mathbf g}$, and hence a left $\Lambda$-module structure.
\item Let $n\ge1$.  For every ordered generating $n$-tuple
$\mathbf g=(g_1,\ldots,g_n)$, there is a natural exact sequence of left
$\Lambda$-modules
\begin{equation}
 0\longrightarrow M_{\mathbf g}
 \longrightarrow\Lambda^n
 \xrightarrow{\partial_{\mathbf g}}I_G
 \longrightarrow0,
\qquad
 \partial_{\mathbf g}(e_i)=g_i-1.
\label{eq:relation-sequence-short}
\end{equation}
\item Let $n\ge1$, and let $\mathbf g$ and $\mathbf h$ be ordered generating
$n$-tuples of $G$.  If $\mathbf g\sim_N\mathbf h$, then
\begin{equation}
 M_{\mathbf g}\cong M_{\mathbf h}
\label{eq:relation-nielsen-short}
\end{equation}
as $\Lambda$-modules.
\item Let $n\ge1$, and let $\mathbf g$ be an ordered generating $n$-tuple of
$G$.  Then there is an isomorphism of $\Lambda$-modules
\begin{equation}
 M_{\operatorname{st}(\mathbf g)}
 \cong M_{\mathbf g}\oplus\Lambda.
\label{eq:relation-stabilization-short}
\end{equation}
\item If $\mathbf g$ and $\mathbf h$ are ordered generating pairs of $G$, then
there is an isomorphism of $\Lambda$-modules
\begin{equation}
 M_{\mathbf g}\oplus\Lambda^2
 \cong M_{\mathbf h}\oplus\Lambda^2.
\label{eq:relation-two-free-short}
\end{equation}
\end{enumerate}
\end{prop}

\begin{proof}
For part~\textup{(a)}, if $\widetilde g$ and $\widetilde g'$ are two lifts of
$g\in G$, then $\widetilde g'=s\widetilde g$ for some $s\in R_{\mathbf g}$.
Conjugation by $s$ acts trivially on the abelianization of $R_{\mathbf g}$, so
$[\widetilde g'r(\widetilde g')^{-1}]=[\widetilde g r\widetilde g^{-1}]$.
Thus the action is independent of the chosen lift.  If lifts of $g$ and $h$
are chosen, their product is a lift of $gh$, which gives the action law; the
identity acts trivially.  Extending $\Z$-linearly yields the stated
$\Lambda$-module structure.

For part~\textup{(b)}, let $\Gamma_{\mathbf g}$ be the Cayley graph whose
vertices are the elements of $G$ and which has, for every $x\in G$ and every
$i$, an oriented edge from $x$ to $xg_i$.  The group $G$ acts on this graph by
left multiplication.  The cellular chain modules of $\Gamma_{\mathbf g}$ are
\[
 C_1(\Gamma_{\mathbf g};\Z)\cong\Lambda^n,
 \qquad
 C_0(\Gamma_{\mathbf g};\Z)\cong\Lambda,
\]
and, with the base edge from $1$ to $g_i$ corresponding to $e_i$, the cellular
boundary satisfies $\partial_{\mathbf g}(e_i)=g_i-1$.  The graph is the
quotient by $R_{\mathbf g}$ of the Cayley tree of $F_n$, so
$H_1(\Gamma_{\mathbf g};\Z)\cong M_{\mathbf g}$.  Since $\mathbf g$
generates $G$, the image of the boundary is $I_G$, proving
\eqref{eq:relation-sequence-short}.

For part~\textup{(c)}, if
$\theta_{\mathbf h}=\theta_{\mathbf g}\circ\Phi$ with $\Phi\in\Aut(F_n)$, then
$\Phi$ maps $R_{\mathbf h}$ isomorphically to $R_{\mathbf g}$ and induces a
$G$-equivariant isomorphism of their abelianizations.

For part~\textup{(d)}, the boundary map for
$\operatorname{st}(\mathbf g)=(g_1,\ldots,g_n,1)$ is
\[
 \partial_{\operatorname{st}(\mathbf g)}
 =\partial_{\mathbf g}\oplus0:
 \Lambda^n\oplus\Lambda\longrightarrow I_G.
\]
Taking kernels gives \eqref{eq:relation-stabilization-short}.

For part~\textup{(e)}, let
\[
 P=\{(x,y)\in\Lambda^2\oplus\Lambda^2
 \mid \partial_{\mathbf g}(x)=\partial_{\mathbf h}(y)\}.
\]
Projection of $P$ onto the first copy of $\Lambda^2$ gives a short exact
sequence
\[
 0\longrightarrow M_{\mathbf h}\longrightarrow P
 \longrightarrow\Lambda^2\longrightarrow0.
\]
The sequence splits because $\Lambda^2$ is free, so
$P\cong M_{\mathbf h}\oplus\Lambda^2$.  Projection onto the second copy gives
similarly $P\cong M_{\mathbf g}\oplus\Lambda^2$.  Comparing the two
decompositions yields \eqref{eq:relation-two-free-short}; this is the usual
pullback proof of Schanuel's lemma.
\end{proof}

\begin{cor}\label{cor:relation-module-obstruction}
Let $G$ be a group, let $n\ge1$, and let $\mathbf g,\mathbf h$ be ordered
generating $n$-tuples of $G$.  If, as $\Z G$-modules,
\begin{equation}
 M_{\mathbf g}\oplus\Z G
 \not\cong
 M_{\mathbf h}\oplus\Z G,
\label{eq:relation-obstruction-short}
\end{equation}
then
\begin{equation}
 \operatorname{st}(\mathbf g)\not\sim_N\operatorname{st}(\mathbf h).
\label{eq:stable-obstruction-short}
\end{equation}
\end{cor}

\begin{proof}
If the stabilized tuples were Nielsen equivalent, parts~(c) and~(d) of
\cref{prop:relation-module-properties} would give an isomorphism contradicting
\eqref{eq:relation-obstruction-short}.
\end{proof}

Equations \eqref{eq:relation-stabilization-short} and
\eqref{eq:relation-two-free-short} mirror the distinction between one and two
stabilizations: one free summand may retain a cancellation obstruction, whereas
two never distinguish relation modules of generating pairs.

\subsection{Examples and limitations}

For $j\ge0$, we use the notation $\operatorname{st}_j$ from
\cref{def:nielsen-equivalence} for $j$-fold stabilization.  Evans's
non-cancellation theorem gives genuine instances of
\cref{cor:relation-module-obstruction}.

\begin{cor}[Evans's non-cancellation examples]\label{cor:evans-stable-nielsen}
For every integer $j\ge1$, there exist a finitely generated group $G$, an
integer $m\ge1$, and ordered generating $m$-tuples $\mathbf g,\mathbf h$ of
$G$ such that
\begin{equation}
 M_{\mathbf g}\oplus(\Z G)^j
 \not\cong
 M_{\mathbf h}\oplus(\Z G)^j.
\label{eq:evans-noncancellation}
\end{equation}
Consequently,
\begin{equation}
 \operatorname{st}_j(\mathbf g)
 \not\sim_N
 \operatorname{st}_j(\mathbf h).
\label{eq:evans-stable-inequivalence}
\end{equation}
\end{cor}

\begin{proof}
Evans \cite{EvansRelationII} proves that for every $j\ge1$ there are a group
$G$ and two presentations of $G$ on the same number $m$ of generators whose
relation modules $R$ and $S$ satisfy
\[
 R\oplus(\Z G)^j\not\cong S\oplus(\Z G)^j.
\]
Evans uses right modules; applying the standard involution
$\Z G\to\Z G$, $g\mapsto g^{-1}$, translates his statement to our left-module
convention.  Let $\mathbf g$ and $\mathbf h$ be the ordered generating
$m$-tuples obtained as the images of the two free bases.  Their relation
modules are $R$ and $S$.  Iterating
\cref{prop:relation-module-properties}(d) gives
\[
 M_{\operatorname{st}_j(\mathbf g)}
 \cong M_{\mathbf g}\oplus(\Z G)^j,
 \qquad
 M_{\operatorname{st}_j(\mathbf h)}
 \cong M_{\mathbf h}\oplus(\Z G)^j.
\]
Thus \eqref{eq:evans-noncancellation} and
\cref{prop:relation-module-properties}(c) imply
\eqref{eq:evans-stable-inequivalence}.
\end{proof}

\begin{rem}\label{rem:evans-concrete}
The case $j=1$ of Theorem~1.1 in \cite{EvansRelationII} already gives a
concrete one-stabilization obstruction: it produces a finitely generated group
$G$ with two finite presentations on the same number of generators whose
relation modules $R,S$ satisfy
\[
 R\oplus\Z G\not\cong S\oplus\Z G.
\]
Thus the associated generating tuples remain Nielsen inequivalent after one
stabilization.  These examples are far from the rank-two examples considered
next.
\end{rem}

The next example shows the opposite phenomenon: relation modules distinguish
generating pairs before stabilization, but the distinction disappears after
one stabilization.

\begin{rem}[The Baumslag--Solitar group $BS(2,3)$]\label{rem:BS23}
Let
\[
 G=BS(2,3)=\langle x,y\mid xy^2x^{-1}=y^3\rangle
\]
and put $z=y^4$.  Harlander and Jensen \cite{HarlanderJensen} show that
$(x,z)$ generates $G$ and that the relation module $M_{(x,z)}$ is stably free
but nonfree of rank one:
\begin{equation}
 M_{(x,z)}\oplus\Z G\cong(\Z G)^2,
 \qquad
 M_{(x,z)}\not\cong\Z G.
\label{eq:BS-exotic-module}
\end{equation}
The standard one-relator presentation is aspherical (see, for example, \cite{HarlanderJensen} and the references there), so
$M_{(x,y)}\cong\Z G$.  Thus \cref{prop:relation-module-properties}(c) distinguishes the generating
pairs $(x,y)$ and $(x,z)$.  Since they share the first coordinate and both
generate $G$, \cref{prop:common-coordinate-stabilization} gives
\begin{equation}
 (x,y,1)\sim_N(x,z,1).
\label{eq:BS-stabilized}
\end{equation}
Accordingly, their relation modules become isomorphic after adjoining one free
summand.  Thus unstabilized non-isomorphism alone does not obstruct one
stabilization.
\end{rem}

Genuine rank-two stable inequivalence nevertheless occurs.  Kapovich and
Weidmann \cite{KWsmall} consider generic small-cancellation presentations of
the form
\begin{equation}
 G=\left\langle a_1,a_2,b_1,b_2\ \middle|\
 a_i=u_i(b_1,b_2),\ b_i=v_i(a_1,a_2),\ i=1,2\right\rangle.
\label{eq:KW-mutual-substitution}
\end{equation}
Both $\mathbf a=(a_1,a_2)$ and $\mathbf b=(b_1,b_2)$ generate $G$, and for
generic choices of the words $u_i,v_i$ they prove
\begin{equation}
 \operatorname{st}(\mathbf a)\not\sim_N
 \operatorname{st}(\mathbf b).
\label{eq:KW-one-stable}
\end{equation}
We do not know whether \cref{cor:relation-module-obstruction} detects these
rank-two examples; they are a natural test case for both cancellation and the
quotient refinement in the following subsection.

\subsection{A quotient relation-module refinement for future applications}

The preceding examples show that the direct obstruction is nontrivial but may
cancel even for a natural two-generated group.  For future applications,
especially to rank-two examples such as \eqref{eq:KW-mutual-substitution}, we
record a quotient-level refinement over a chosen quotient of $G$.

Let
\begin{equation}
 \pi:G\twoheadrightarrow Q,
 \qquad K=\ker\pi,
\label{eq:quotient-setup-short}
\end{equation}
let $n\ge1$, and let $\mathbf g=(g_1,\ldots,g_n)$ be an ordered
generating $n$-tuple of $G$.  Write
\[
 \overline{\mathbf g}
 =(\pi(g_1),\ldots,\pi(g_n)).
\]
If
\[
 R=\ker \theta_{\mathbf g},
 \qquad
 S=\ker \theta_{\overline{\mathbf g}},
\]
then restriction of $\theta_{\mathbf g}$ gives
\begin{equation}
 1\longrightarrow R\longrightarrow S\longrightarrow K\longrightarrow1.
\label{eq:R-S-K-short}
\end{equation}
Conjugation by $G$ on both $K_{\rm ab}$ and $H_2(K;\Z)$ factors through
$Q$, since inner automorphisms of $K$ act trivially on these groups.  Thus both
are naturally left $\Z Q$-modules.  The map $S\to K$ induces a natural
$\Z Q$-module epimorphism
\begin{equation}
 \theta_{\pi,\mathbf g}:
 M_{\overline{\mathbf g}}=S/[S,S]
 \longrightarrow K_{\rm ab}.
\label{eq:theta-short}
\end{equation}
For a relation $w(\overline{\mathbf g})=1$ in $Q$,
\begin{equation}
 \theta_{\pi,\mathbf g}([w])
 =[w(\mathbf g)]\in K_{\rm ab}.
\label{eq:theta-evaluation-short}
\end{equation}
Put
\begin{equation}
 C_{\pi,\mathbf g}=\ker\theta_{\pi,\mathbf g}.
\label{eq:C-short}
\end{equation}

We regard $\Z Q$ as a $(\Z Q,\Z G)$-bimodule via $\pi$, so that
$q\cdot g=q\pi(g)$ for $q\in Q$ and $g\in G$.  The standard
five-term homology sequence gives the following refinement; see Brown
\cite{BrownCohomology}.

\begin{prop}\label{prop:refined-relation-obstruction}
Let $\pi:G\twoheadrightarrow Q$ have kernel $K$, let $n\ge1$, and let
$\mathbf g$ be an ordered generating $n$-tuple of $G$.
\begin{enumerate}[label=\textup{(\alph*)}]
\item There is a natural exact sequence of $\Z Q$-modules
\begin{equation}
 0\longrightarrow H_2(K;\Z)
 \longrightarrow
 \Z Q\otimes_{\Z G}M_{\mathbf g}
 \longrightarrow
 M_{\overline{\mathbf g}}
 \xrightarrow{\theta_{\pi,\mathbf g}}
 K_{\rm ab}
 \longrightarrow0.
\label{eq:refined-five-term-short}
\end{equation}
\item Stabilization gives an isomorphism of $\Z Q$-modules
\begin{equation}
 C_{\pi,\operatorname{st}(\mathbf g)}
 \cong C_{\pi,\mathbf g}\oplus\Z Q.
\label{eq:C-stabilization-short}
\end{equation}
\item If $\mathbf g,\mathbf h$ are ordered generating $n$-tuples of $G$
and, as $\Z Q$-modules,
\begin{equation}
 C_{\pi,\mathbf g}\oplus\Z Q
 \not\cong
 C_{\pi,\mathbf h}\oplus\Z Q,
\label{eq:refined-obstruction-short}
\end{equation}
then
\begin{equation}
 \operatorname{st}(\mathbf g)\not\sim_N\operatorname{st}(\mathbf h).
\label{eq:refined-conclusion-short}
\end{equation}
\end{enumerate}
\end{prop}

\begin{proof}
For part~(a), apply the five-term homology sequence to
\eqref{eq:R-S-K-short}.  Since $R$ and $S$ are normal in $F_n$, the sequence is
$Q=F_n/S$-equivariant.  Moreover, $S$ is a subgroup of the free group $F_n$,
so $H_2(S;\Z)=0$.  Hence there is an exact sequence of $\Z Q$-modules
\[
 0\to H_2(K;\Z)\to R/[R,S]
 \to S/[S,S]\to K_{\rm ab}\to0.
\]
The quotient $R/[R,S]$ is the module of $K$-coinvariants of
$M_{\mathbf g}=R/[R,R]$.  With the bimodule convention above,
\[
 R/[R,S]
 \cong (M_{\mathbf g})_K
 \cong \Z Q\otimes_{\Z G}M_{\mathbf g}.
\]
This gives \eqref{eq:refined-five-term-short}.  By exactness, it also gives
\[
 C_{\pi,\mathbf g}
 \cong
 \frac{\Z Q\otimes_{\Z G}M_{\mathbf g}}
 {\operatorname{im}\bigl(H_2(K;\Z)\to
  \Z Q\otimes_{\Z G}M_{\mathbf g}\bigr)}
\]
as $\Z Q$-modules.

For part~(b),
\[
 M_{\operatorname{st}(\overline{\mathbf g})}
 \cong M_{\overline{\mathbf g}}\oplus\Z Q
\]
by \cref{prop:relation-module-properties}(d), and the additional free relation
maps to $0$ in $K_{\rm ab}$.  Hence
$\theta_{\pi,\operatorname{st}(\mathbf g)}=\theta_{\pi,\mathbf g}\oplus0$,
which gives \eqref{eq:C-stabilization-short}.

For part~(c), suppose the stabilized tuples are Nielsen equivalent.  If the
corresponding Nielsen automorphism is $\Phi$, then $\Phi$ carries the relation
subgroup for $\operatorname{st}(\mathbf h)$ to that for
$\operatorname{st}(\mathbf g)$ and likewise carries the relation subgroup of
the induced presentation of $Q$ to the corresponding one.  Moreover, for an
element $s$ of the latter subgroup,
\[
 \theta_{\operatorname{st}(\mathbf g)}(\Phi(s))
 =\theta_{\operatorname{st}(\mathbf h)}(s)\in K.
\]
If $\Phi_*$ denotes the induced isomorphism of quotient relation modules,
then the preceding identity says precisely that
\[
 \theta_{\pi,\operatorname{st}(\mathbf g)}\circ\Phi_*
 =\theta_{\pi,\operatorname{st}(\mathbf h)}.
\]
Thus $\Phi_*$ restricts to an isomorphism of the kernels of the two evaluation
maps.  Part~(b) then contradicts \eqref{eq:refined-obstruction-short}.
\end{proof}

\begin{rem}\label{rem:johnson-direction}
For
\[
 \rho:\Out(F_n)\twoheadrightarrow\GL(n,\Z),
\]
the kernel is the outer IA-group $IO_n$.  Thus a generating pair
$\mathbf g$ with linear image $(A,B)$ determines
\begin{equation}
 \theta_{\rho,\mathbf g}:
 M_{(A,B)}\twoheadrightarrow (IO_n)_{\rm ab}.
\label{eq:johnson-direction-short}
\end{equation}
The target $(IO_n)_{\rm ab}=H_1(IO_n;\Z)$ is the natural linearized quotient
of the outer IA-group.  Formula \eqref{eq:johnson-direction-short} suggests
searching for stable Nielsen invariants through this abelianization or through
congruence quotients, rather than directly over $\Z\Out(F_n)$.  We do not
pursue this here.
\end{rem}

The families in \cref{thm:main} realize neither obstruction, since
\cref{cor:families-stabilize} gives actual equivalences of their stabilized
triples.  \Cref{cor:evans-stable-nielsen} shows that the direct obstruction is
effective for other groups.  We do not presently know an example where the
quotient refinement separates once-stabilized tuples not already separated by
the direct relation module; we record it as a potentially more computable tool.

\section{Outlook and open problems}\label{sec:questions}

The bound $n\ge 8$ in \cref{thm:linear-nielsen,thm:main} comes from the
disjoint supports used in the construction and is not expected to be optimal.

\begin{qst}\label{qst:small-ranks}
For which ranks $3\le n\le7$ do $\Out(F_n)$ and $\SOut(F_n)$ have infinitely
many Nielsen equivalence classes of generating pairs?
\end{qst}

The main stabilization question can be formulated intrinsically in terms of
weakly redundant generating triples.

\begin{qst}\label{qst:redundant-connectivity}
Let $n\ge 3$.
\begin{enumerate}[label=\textup{(\alph*)}]
\item Is the graph $N_3^{\rm wr}(\Out(F_n))$ connected?
\item Is the graph $N_3^{\rm wr}(\SOut(F_n))$ connected?
\end{enumerate}
\end{qst}

For a two-generated group, this connectivity is equivalent to asking whether
any two generating pairs become Nielsen equivalent after one stabilization.  A
negative answer would give stable Nielsen invariants for the corresponding
outer automorphism group.

\begin{qst}\label{qst:relation-cancellation}
We do not know the answers to the following relation-module cancellation problems.
\begin{enumerate}[label=\textup{(\alph*)}]
\item Does there exist a two-generated group $G$ and ordered generating pairs
$\mathbf g,\mathbf h$ such that, as $\Z G$-modules,
\[
 M_{\mathbf g}\oplus\Z G
 \not\cong
 M_{\mathbf h}\oplus\Z G?
\]
\item Let $n\ge2$, and let $G$ be $\Out(F_n)$ or $\SOut(F_n)$.  Do there
exist ordered generating pairs $\mathbf g,\mathbf h$ of $G$ satisfying the
same non-isomorphism?
\item For such a group $G$, do there exist an epimorphism
$\pi:G\twoheadrightarrow Q$ and ordered generating pairs
$\mathbf g,\mathbf h$ such that, as $\Z Q$-modules,
\[
 C_{\pi,\mathbf g}\oplus\Z Q
 \not\cong
 C_{\pi,\mathbf h}\oplus\Z Q?
\]
\end{enumerate}
\end{qst}

\Cref{cor:evans-stable-nielsen} gives non-cancellation after adjoining one free summand for
longer tuples, whereas \eqref{eq:KW-one-stable} gives the corresponding
rank-two group-theoretic inequivalence.  Thus
\cref{qst:relation-cancellation}(a) isolates a natural gap.  Parts~(b) and~(c)
provide algebraic approaches to the negative direction in
\cref{qst:redundant-connectivity}; conversely, weakly redundant connectivity
would force these one-step obstructions to vanish for the outer automorphism
groups.

\section{Disclosure of AI use}

AI-assisted tools were used in preparing preliminary drafts.  The author has
checked and edited the text and takes full responsibility for the paper.

\end{document}